\documentclass[letterpaper,11pt]{article}

\usepackage{amsthm}
\usepackage{amsmath}

\newtheorem{thm}{Theorem}[section]

\newtheorem{lemma}[thm]{Lemma}
\newtheorem{prop}[thm]{Proposition}

\newtheorem{conj}[thm]{Conjecture}

\newcommand{\beq}[1]{\begin{equation}\label{#1}}
\newcommand{\enq}[0]{\end{equation}}

\newcommand{\bn}[0]{\bigskip\noindent}
\newcommand{\mn}[0]{\medskip\noindent}
\newcommand{\nin}[0]{\noindent}

\newcommand{\sub}[0]{\subseteq}
\newcommand{\sm}[0]{\setminus}
\renewcommand{\dots}[0]{,\ldots,}

\newcommand{\A}[0]{{\cal A}}

\newcommand{\EE}[0]{{\bf E}}

\newcommand{\ww}{\mbox{{\sf w}}}

\newcommand{\E}[0]{{\bf E}}
\newcommand{\0}[0]{\emptyset}

\renewcommand{\qed}[0]{\begin{flushright} \rule{2mm}{3mm} \end{flushright}}

\def\qqqed{\null\nobreak\hfill\hbox{\rule{2mm}{3mm} }\par\smallskip}

\newcommand{\C}[2]{{{#1}\choose{{#2}}}}
\newcommand{\Cc}[0]{\tbinom}
\newcommand{\ga}[0]{\alpha }
\newcommand{\gb}[0]{\beta }
\newcommand{\gc}[0]{\gamma }
\newcommand{\gd}[0]{\delta }
\newcommand{\gD}[0]{\Delta }

\newcommand{\gl}[0]{\lambda }

\newcommand{\go}[0]{\omega}
\newcommand{\gO}[0]{\Omega}

\newcommand{\gs}[0]{\sigma}

\newcommand{\gz}[0]{\zeta}
\newcommand{\eps}[0]{\varepsilon }
\newcommand{\vt}[0]{\vartheta}

\newcommand{\rrr}[0]{r}

\newcommand{\comments}[1]{}
\usepackage[usenames]{color}
\begin{document}
\renewcommand{\thefootnote}{\fnsymbol{footnote}}

\title{Tight upper tail bounds for cliques}
\author{B. DeMarco\footnotemark $~$ and J. Kahn\footnotemark
}
\date{}
\footnotetext{ * supported by the U.S.
Department of Homeland Security under Grant Award Number 2007-ST-104-000006.}
\footnotetext{ $\dag$ Supported by NSF grant DMS0701175.}

\date{}

\maketitle

\begin{abstract}
With $\xi_{k}=\xi_{k}^{n,p}$ the number of copies of $K_k$ in the usual (Erd\H{o}s-R\'enyi) random
graph $G(n,p)$, $p\geq n^{-2/(k-1)}$ and $\eta>0$, we show when $k>1$
$$\Pr(\xi_k> (1+\eta)\E \xi_k) < \exp \left[-\gO_{\eta,k} \min\{n^2p^{k-1}\log(1/p), n^kp^{\binom{k}{2}}\}\right].$$
This is tight up to the value of the constant in the exponent.

\end{abstract}

\section{Introduction}
Let $G(n,p)$ be the Erd\H{o}s-R\'enyi random graph on $n$ vertices,
in which every edge occurs independently with probability $p$, and let $H$ be a fixed graph with $v_H=|V(H)|$ and $e_H=|E(H)|$.
A {\it copy} of $H$ in $G(n,p)$ is any subgraph of $G(n,p)$ isomorphic to $H$.  It has been a long studied question (e.g. \cite{Chat,DeM,JOR,Jan4,Jan5, Kim2, Vu1})
to estimate,
for $\eta>0$ and $\xi_H=\xi_H^{n,p}$ the number of copies of $H$ in $G(n,p)$,
\beq{Prxi}
\Pr\left(\xi_H>(1+\eta)\E \xi_H\right).
\enq

\medskip
To avoid irrelevancies, let us declare at the outset that we always assume
$p\geq n^{-1/m_H}$, where, as usual
(e.g. \cite[p.6]{JLR}),
\beq{mH}
m_H =\max\{e_K/v_K:K\sub H\}
\enq
(so $n^{-1/m_H}$ is a threshold for ``$G\supseteq H$";
see \cite[Theorem 3.4]{JLR});
in particular, when $H=K_k$ we assume $p\geq n^{-2/(k-1)}$.
For smaller $p$
the problem is not very interesting
(e.g. for bounded $\eta$ the probability in \eqref{Prxi} is easily seen to be
$\Theta(\min\{n^{v_{K}}p^{e_{K}}:K\sub H, e_{K}>0\})$;
see \cite[Theorem 3.9]{JLR} for a start),
and we will not pursue it here.

Janson and Ruci\'nski $\cite{Jan4}$
offer a nice overview of the methods used prior to 2002 to
obtain upper bounds on the probability in \eqref{Prxi},
by far the more challenging part of the problem.
To get an idea of the difficulty, note that
even for the case that $H$ is a triangle, only quite poor upper bounds
were known until a breakthrough result of Kim and Vu \cite{Kim2},
who used, {\em inter alia},
the ``polynomial concentration" machinery of \cite{KimVu} to show,
for $p>n^{-1}\log n$,
\beq{kimvu}
\exp_p[O_\eta(n^2p^2)] < \Pr(\xi_H>(1+\eta)\E \xi_H) < \exp[-\gO_\eta(n^2p^2)].
\enq

\mn
(The easy lower bound, seemingly first observed in \cite{Vu1},
is, for example, the probability of containing a complete graph on
something like
$(1+\eta)^{1/3}np$ vertices.  Of course the subscript $\eta$ in the lower
bound is unnecessary if, for example, $\eta\leq 1$, which is what we usually
have in mind.)  Polynomial concentration was also used by Vu
\cite{Vu} to show that if $H$ is strictly balanced and
$\E \xi_H\leq\log n$, then
\beq{vu}
\Pr(\xi_H>(1+\eta)\E \xi_H)< \exp[-\gO_{\eta}(\E \xi_H)].
\enq

\medskip
The result of \cite{Kim2} was vastly extended in a beautiful paper of
Janson, Oleszkiewicz and Ruci\'nski \cite{JOR}, where it was shown
that for any $H$ and $\eta$,
\beq{jor}
\exp_p[O_{H,\eta}(M_H(n,p))]< \Pr(\xi_H> (1+\eta)\E \xi_H) <
\exp [-\gO_{H,\eta}(M_H(n,p))],
\enq
thus determining the probability \eqref{Prxi} up to a factor $O(\log(1/p))$
in the exponent for constant $\eta$.
A definition of $M$ is given in Section \ref{CR}; for now we
just mention that (for $p\geq n^{-2/(k-1)}$)
$M_{K_k}(n,p)=n^2p^{k-1}$.  (It should also be noted that in the limited
range where it applies, the upper bound
in \eqref{vu} is better than the one in \eqref{jor}.)

\medskip
While it seems natural to expect that the lower bound in \eqref{jor}
is ``usually" the truth (see Section \ref{CR} for a precise guess),
the only progress
in this direction until quite recently was
\cite{Jan5}, which
established the upper bound
$
\exp [-\gO(M_H(n,p)\log^{1/2}(1/p))]
$
for $H=K_4$ or $C_4$ (the 4-cycle)
and {\em some} values of $p$.

The $\log(1/p)$ gap was finally closed for the case $H=K_3$ by Chatterjee
\cite{Chat} and, independently, the present authors \cite{DeM}.
More precisely, \cite{Chat}
showed that for a suitable $C$ depending on $\eta$ and
$p>Cn^{-1}\log n$,

$$\Pr(\xi_{K_3}> (1+\eta)\E \xi_{K_3}) < p^{\Omega_{\eta}(n^2p^2)},$$
while \cite{DeM} showed, somewhat more generally, that for $p>n^{-1}$,
$$\exp [-O_{\eta}(f(3,n,p))]<\Pr(\xi_{K_3}> (1+\eta)\E \xi_{K_3})
<\exp [-\Omega_{\eta}( f(3,n,p))],$$
where $f(k,n,p):=\min\{n^2p^{k-1}\log(1/p), n^kp^{\binom{k}{2}}\}$.
(In what follows we will often abbreviate $f(k,n,p)=f(k,n)$.)

In this paper we considerably extend the method
of \cite{DeM} to settle the problem for general cliques and a bit more.

\begin{thm}\label{upper}
Assume $H$ on $k$ vertices has minimum degree at least $k-2$ (that is, the complement of $H$ is a matching).  Then for all $\eta>0$ and
$ p\geq n^{-2/(k-1)}$,
  $$\Pr\left(\xi_H\geq (1+\eta)\E(\xi_H)\right)\leq \exp \left[-\Omega_{\eta,H}(f(k,n,p))\right].$$
\end{thm}
\mn
\begin{thm}\label{lower}
For $H=K_k$ and for all $p\geq n^{-2/(k-1)}$,
  $$\Pr\left(\xi_H\geq 2\E(\xi_H)\right)\geq \exp \left[-O_{H}(f(k,n,p))\right].$$
\end{thm}

\nin
{\em Remarks.}
1.
We are most interested in the ``nonpathological" range where
$f(k,n,p)=n^2p^{k-1}\log(1/p)$, so when
$p\geq n^{-2/(k-1)}(\log n)^{2/[(k-1)(k-2)]}$ (or a bit less).
It may be helpful to think mainly of this range as we proceed.

\mn
2.
Though mainly concerned with the case $H=K_k$ in Theorem \ref{upper},
we prove the more general statement for inductive reasons.
For noncliques the bound of Theorem \ref{upper} is not usually tight;
more precisely:
it is tight (up to the constant in the exponent)
if $p=\gO(1)$ or if $\gD:=\gD_H=k-1$ and $p=\gO(n^{-1/\gD})$,
in which cases our upper bound agrees with the lower bound in \eqref{jor};
it is not tight if $\gD=k-2$ and $p=o(1)$
(see the proof of Lemma \ref{dk2})
or if $H\neq K_k$ and $p< n^{-c/\gD}$ for some fixed $c>1$
(see the proof of Lemma \ref{lowerp};
in fact $p=o(n^{-1/\gD})$ is probably enough here---which
would complete this little story---but we don't quite show this).

\medskip
In the next section we show that Theorem \ref{upper} follows from
an analogous assertion for $k$-partite graphs;
most of the paper (Sections \ref{LD}-\ref{LP5})
is then concerned with this modified
problem.  Section \ref{PT1.2} gives the proof of Theorem \ref{lower}
and Section \ref{CR} contains a few concluding remarks.

\section{Reduction}\label{Reduction}

For the rest of the paper we set $t=\log (1/p)$ and take $H$ to be a graph with vertices $v_1,v_2,\ldots, v_k$.  We define $ G=G(n,p,H)$ to be the random graph with vertex set $V=V_1 \cup \cdots \cup V_k$, where the $V_i$'s are disjoint $n$-sets and $\Pr(xy\in E(G))=p$ whenever $x\in V_i$ and $y\in V_j$ for some $v_iv_j\in E(H)$, these choices made independently.  We define a {\it copy} of $H$ in $G$ to be a set of vertices $\{x_1,\ldots,x_k\}$ with $x_i\in V_i$ and $x_ix_j\in E(G)$ whenever $v_iv_j\in E(H)$; use $X_H^{n,p}$ for
the number of such copies; and set
$\Psi(H,n,p)=\E(X_H^{n,p})=n^kp^{e_H}$.
When there is no danger of confusion we will often use $X_H^n$---or,
for typographical reasons $X(H,n)$---for $X_H^{n,p}$
and $\Psi(H,n)$ for $\Psi(H,n,p)$.

The next two propositions show an equivalence between $G(n,p)$ and $G$
with regard to upper tails for subgraph counts. In each we set
$\ga = |{\rm Aut}(H)|n^k/(kn)_k \sim k^{-k}|{\rm Aut}(H)|$
(where as usual $(a)_b=a(a-1)\cdots (a-b+1)$).

\begin{prop}\label{Equiv1}
For $\eta>0$ and $\eps=\eta/(2+\eta)$,
$$\Pr(X_H^{n,p}\geq (1+\eps)\Psi(H,n,p))
>\frac{\alpha \eps}{1-\ga+\alpha\eps}
\Pr(\xi_H^{kn,p}\geq (1+\eta)\E
(\xi_H^{kn,p}) )$$
\end{prop}
\nin We omit the proof of Proposition \ref{Equiv1} since it is a straightforward
generalization of the case $H=K_3$ proved in \cite{DeM}.

\begin{prop}\label{Equiv2}
For any $\eps>0$ there is a $C=C_{\eps,H}$ such that for $p>Cn^{-1/m_H}$,
$$\Pr\left(X_H^{n,p}\geq (1+\eps)\Psi(H,n,p)\right)
< 2\Pr(\xi_H^{kn,p}\geq (1+\alpha\eps/2)\E (\xi_H^{kn,p})) .$$
\end{prop}
\mn
(See \eqref{mH} for $m_H$.)

\begin{proof}
We may choose $G^*=G(kn,p)$ by first choosing ${G}={G}(n,p,H)$ and then letting

$$E(G^*)=E({G}) \cup S$$
where $\Pr(xy\in S)=p$ whenever $x\neq y$, $x\in V_i$ and $y\in V_j$ for some $v_iv_j\not\in E(H)$, these choices made independently.
Write $\xi$ and $X$ for
the numbers of copies of $H$ in $G^*$
and $G$ respectively
(thus
$ \xi=\xi_H^{kn,p}$ and $X=X_H^{n,p}$), and set $\xi^*=\xi-X$.
Since $\EE X = \ga \EE \xi$,
we have, using Harris' Inequality,
\beq{har}
\Pr(\xi > (1+ \tfrac{\ga\eps}{2})\EE \xi)\geq
\Pr(X > (1+\eps)\EE X)\Pr(\xi^*  > \EE\xi^* - \tfrac{\ga\eps}{2}\EE \xi);
\enq

\mn
so we need to say that the second probability on the right is at least 1/2.
This is standard, but we summarize the argument for completeness.

A result of Janson from \cite{Janson} (see \cite[(2.14)]{JLR}) gives 
\beq{jan}
\Pr(\xi^* \leq \EE\xi^* - t)
< \exp[-\tfrac{t^2}{2\bar{\gD}}],
\enq
with
\beq{gDbar}
\mbox{$
\bar{\gD}= \sum_{\gs\sim \tau}^*\EE I_{\gs}I_{\tau} \leq \
\sum_{\gs\sim \tau}\EE I_{\gs}I_{\tau}$},
\enq
where (recycling notation a little)
$H_1,\ldots $ are the copies of $H$ in $K_{kn}$;
$I_{\gs} ={\bf 1}_{\{H_{\gs}\sub G^*\}}$;
``$\gs\sim \tau$" means $H_{\gs}$ and $H_{\tau}$ share an edge
(so $\gs\sim \gs$);
and $\sum^*$ means we sum only over $\gs,\tau$ for which $H_\gs,H_\tau$
cannot appear in $G$.

\mn
But (very wastefully),

\begin{align*}
\bar{\gD}&< n^{v_H}\sum\{n^{v_H-v_K}p^{2e_H-e_K}: K\sub H,e_K>0\}\\
&< n^{2v_H}p^{2e_H} \sum\{n^{-v_K}(Cn^{-1/m_H})^{-e_K}:K\sub H, e_K>0\}\\
&
= O(C^{-1}\EE^2\xi),
\end{align*}
where $C$ is the constant from \eqref{plb},
which may be taken large compared to the
implied constant in ``$O(\cdot )$."
Thus, using \eqref{jan}  with the above bound on $\bar{\gD}$
and $t=(\ga\eps/2)\EE \xi$, we find that
the second probability on the right side of \eqref{har}
is at least $1-\exp[-\gO((\ga\eps)^2C)]>1/2$.
\end{proof}

According to Proposition \ref{Equiv1}, Theorem \ref{upper} will follow from the corresponding k-partite statement, {\em viz.}

\begin{thm}\label{GGThm}
If $H$ has minimum degree at least $k-2$, then

\mn
{\rm (a)}
for all $\eps>0$,
$$
\Pr\left(X_H^{n,p}\geq (1+\eps)\Psi(H,n,p)\right)< \exp \left[-\Omega_{H,\eps}\left(f(k,n,p)\right)\right];
$$
\mn
{\rm (b)}
for any $\tau\geq 1$,
$$
\Pr\left(X_H^{n,p}\geq 2\tau\Psi(H,n,p)\right)<
\exp[-\Omega_{H}( f(k,n\tau^{1/k},p))]. 
$$
\end{thm}

\bigskip

\nin
Note that (b) for a given $H$ follows from (a), since
(noting that $\tau\Psi(H,n)=\Psi(H,n\tau^{1/k})$ and
using (a) for the second inequality)
\begin{align*}
\Pr\left(X_H^{n}\geq 2\tau\Psi(H,n)\right)&\leq \Pr\left(X_H^{n\tau^{1/k}}\geq 2\Psi(H,n\tau^{1/k})\right)\\
&\leq \exp\left[-\Omega_{H}\left(f(k,n\tau^{1/k},p)\right)\right].
\end{align*}
We include (b) because it will be needed for induction;
that is, for a given $H$ we just prove (a), occasionally
appealing to earlier cases of (b).

\medskip
We have formulated the theorem for all $p$ so that the inductive parts of
the proof don't require checking that $p$ falls in some suitable range.
Note, however, that for the proof we can assume
(for our choice of
positive constants $C$ and $c$ depending on $H$ and $\eps$)
\beq{plb}
p>C n^{-2/(k-1)},
\enq
since for smaller $p$ ($>n^{-1/m_H}$) the theorem is trivial,
and
\beq{pub}
p< c,
\enq
since above this the desired bound is given by
\eqref{jor}.
As detailed in the next two lemmas,
\eqref{jor}, together with some auxiliary results from \cite{JOR},
also allows us to ignore certain other cases of
Theorem \ref{GGThm}(a).

\begin{lemma}\label{dk2}
If $\Delta_H\leq k-2$ then
$$\Pr\left(X_H^{n,p}\geq (1+\eps)\Psi (H,n,p)\right)\leq p^{\Omega_{H,\eps}(n^2p^{k-1})} .$$
\end{lemma}

\begin{proof}
By Proposition \ref{Equiv2}, it is enough to show
\beq{k-2}
\Pr\left(\xi_H^{n,p}\geq (1+\eps)\E(\xi_H^{n,p})\right)\leq p^{\Omega_{H,\eps}(n^2p^{k-1})};
\enq
but this follows from \eqref{jor}, which
since
$M_H(n,p)\geq n^2p^{\Delta_H}$ (see \cite[Lemma $6.2$]{JOR}),
bounds the left side of \eqref{k-2} by
$$
\exp[-\Omega_{H,\eps} (n^2p^{\Delta_H})]\leq
\exp [-\Omega_{H,\eps}(n^2p^{k-1}t)].$$
\end{proof}

\begin{lemma}\label{lowerp}
For any $H\neq K_k$ on $k$ vertices and
$\gc>0$, if $p<n^{-(1+\gc)/(k-1)}$ then
$$\Pr(X_H^{n}\geq (1+\eps)\Psi(H,n))< p^{\gO_{H,\eps,\gc}(n^2p^{k-1})}.$$
\end{lemma}
\nin
\begin{proof}
By Lemma \ref{dk2} we may assume $\gD:=\gD_H = k-1$
(and will write $\gD$ in place of $k-1$ in this section).
By Proposition \ref{Equiv2} it's enough to show
$$\Pr(\xi_H^{n,p}\geq (1+\vt)\E(\xi_H^{n,p}))
< p^{\gO_{\vt,H}(n^2p^{\gD })},$$

\mn
which, in view of \eqref{jor} and the definition of $M_H(n,p)$, will follow if
we show that, for any $K\sub H$,
$n^{v_K}p^{e_K}=\gO((n^2p^{\gD }t)^{\ga^*_K})$, or, more conveniently,
\beq{nvK}
n^{v_K-2\ga^*_K}p^{e_K-\gD \ga^*_K}=\gO(t^{\ga^*_K}).
\enq

We need one easy observation from \cite{JOR} (see their Lemma 6.1):
$$
e_K\leq \gD(v_K-\ga^*_K).
$$

\mn
Then, noting that
\beq{eK}
e_K-\gD \alpha_K^*<0
\enq
(since $e_K< \gD v_K/2\leq \gD \alpha_{K}^{*}$) and
using our upper bound on $p$,
we find that the
left side of \eqref{nvK} is at least
\begin{eqnarray*}
n^{v_K-2\ga^*_K -(1+\gc)(e_K-\gD\ga^*_K)/\gD }
&\geq &
n^{v_K-2\ga^*_K  - (v_K-2\ga^*_K )
+\gc (\gD\ga^*_K-e_K)/\gD }\\
&=&n^{\gc (\gD\ga^*_K-e_K)/\gD },
\end{eqnarray*}
which (again using \eqref{eK}) gives \eqref{nvK}.
\end{proof}

\section{Large deviations}\label{LD}

This section collects a few standardish large deviation basics
that will be used throughout the paper.
It's perhaps worth noting that these elementary inequalities
are the only ``machinery" we will need.

We use $B(m,\ga)$ for a random variable with the binomial distribution ${\rm Bin}(m,\ga)$.
The next lemma, which is easily derived from \cite[Theorem A.1.12]{Alon2} and \cite[Theorem 2.1]{JLR} respectively (for
example),
will be used repeatedly, eventually without explicit mention.

\begin{lemma}\label{L1}
There is a fixed $C>0$ so that for any $\lambda\leq 1$, $K>1+\lambda$, m and $\ga$,
\beq{bin}
\Pr(B(m,\ga) \geq Km\ga) <\min\{(e/K)^{Km\ga}, \exp[-C\lambda^2Km\ga]\}.
\enq
\end{lemma}

\nin
{\em Remark.}
We may assume $Km\ga \geq 1$.  Thus, if $em\ga^{c}<1$ then
$e/K<\ga^{1-c}$ and the bound in \eqref{bin} is at most
$\ga^{(1-c)Km\ga}$.

\medskip
The next lemma, an immediate consequence of Lemma \ref{L1} (and the above Remark),
will also be used repeatedly,
usually following a preliminary application of Lemma \ref{L1} to
justify the assumption $enq^c<1$.

\begin{lemma}\label{L2}
Fix $c<1$ and assume $enq^c<1$.
If
$S\subseteq V_i$ is random with $\Pr(x\in S)\leq q ~\forall x\in V_i$,
these events independent, then for any $T$,
$$\Pr(|S|\geq T)<q^{(1-c)T}.$$
\end{lemma}

We also need the following inequality, which is an easy consequence of, for example, \cite[Lemma 8.2]{Beck-Chen}.

\begin{lemma}\label{Chern}
Suppose $w_1\dots w_m \in [0,z]$.
Let $\xi_1\dots \xi_m$ be independent Bernoullis,
$\xi = \sum \xi_i w_i$, and $\E\xi =\mu$.
Then for any $\eta >0$ and $\gl\geq \eta \mu$,
$$\Pr(\xi > \mu+\gl) < \exp [- \gO_\eta(\gl/z)].$$
\end{lemma}

\section{Outline}

In this section we list the steps in
the proof of Theorem \ref{GGThm}(a),
filling in some definitions as we go along.
The proof proceeds
by induction on (say)
$k^2 + e_H$, so that in proving the statement for $H$
we may assume its truth for all graphs with either fewer than $k$ vertices
or with $k$ vertices and fewer than $e_H$ edges.
The case $k=2$ is trivial and $k=3$ is the main result of \cite{DeM},
so we assume throughout that $k\geq 4$.

Most of the proof (Lemmas \ref{P1}-\ref{P4})
consists of identifying certain anomalies,
for example vertices of unusually high degree,
and bounding the number of copies of $H$ in which they appear.
The remaining copies are then easily handled
(in Lemma \ref{P5}) using Lemma \ref{Chern}.

Here and throughout we use $C$ and $C_\eps$ for (positive) constants depending on (respectively) $H$ and $(H,\eps)$, different occurrences of which will
usually denote different values.
Similarly, we use $\Omega$ and $\Omega_{\eps}$ as shorthand for
$\Omega_H$ and $\Omega_{H,\eps}$.  We say an event $E$ occurs {\em with large probability} (w.l.p.)
if $\Pr(E) > 1-\exp [-\Omega_{\eps} (n^2p^{k-1}t)]$, and write ``$\alpha<^*\beta$" for ``w.l.p. $\alpha<\beta$"
(where $\eps$ is as in the statement of the theorem).
Note that \eqref{plb} (with a suitable $C$) guarantees that an intersection of,
for example, $n^5$ w.l.p. events is itself a w.l.p. event,
a fact we will sometimes use without mention in what follows.

By Lemma \ref{dk2} we may assume $\Delta_H=k-1$.
We reorder the vertices of $H$ so that
$k-1=d(v_1)\geq d(v_2)\geq \ldots \geq d(v_k)$ and if $d(v_2)=k-2$ then $v_2\not\sim v_3$.  We set $A=V_1, B=V_2,C=V_3$ and {\em always take
$a,b$ and $c$ to be elements of
$A,B$ and $C$ respectively.}
For disjoint $X,Y\sub V$ we use $\nabla(X,Y)$ for the set of edges with one end in each of $X$ and $Y$, and $\nabla(X)$ for the set of edges with one end in $X$.
We use $N(x)$ for the neighborhood of (set of vertices adjacent to) a
vertex $x$.

For $K\subseteq H$ with vertex set $\{v_i:i\in T\}$ ($T\subseteq [k]$),
define a {\it copy of $K$} in $G$ ($=G(n,p,H)$)
to be a set of vertices $\{x_i:i\in T\}$
with $x_i \in V_i$ and $x_ix_j\in E(G)$ whenever $v_iv_j\in E(K)$.
For $x_1,x_2,\ldots, x_l$ vertices belonging to distinct $V_i$'s we
use $\ww_{K}(x_1,\ldots, x_l)$ for the number of copies of $K$
containing $x_1,\ldots,x_l$; when $K=H$ we call this the
{\em weight} of $\{x_1,\ldots,x_l\}$.
We use $H_{S}=H-\{v_i:i\in S\}$ ($S\subset [k]$),
and abbreviate $H_{\{i\}}=H_i$,
$\ww_{H_S}(\cdot)=\ww_{S}(\cdot)$, $\ww_{\{i\}}=\ww_i$
and $\ww_{\emptyset}(\cdot)$ ($=\ww_H(\cdot)$) $=\ww(\cdot)$.

Set $\vt=.05 \eps$ and
define $\delta$ by $(1+\delta)^k=2$.  For $x\in V$ and $i\in [k]$,
let $d_i(x)=|N(x)\cap V_i|$, and set $d(x)=\max\{d_i(x):i\in [k]\}$.
Say a vertex $x$ is {\it high degree} if $d(x)>(1+\delta)np$, and a copy of $H$
is {\it type one} if contains a high degree vertex from $A,B$ or $C$.

\begin{lemma}\label{P1}
 W.l.p. G contains less than $7\vt \Psi(H,n)$ type one copies of $H$.
\end{lemma}

Let $A', B', C'$ denote the subsets of $A,B,C$ respectively of vertices which
are not high degree.  For vertices $x,y\in G$ let
$d_j(x,y)=|N(x)\cap N(y)\cap V_j|$ and $d(x,y)=\max_{j\geq 4}d_j(x,y)$.
A pair of vertices $(x,y)$ is
{\it high degree} if $d(x,y)>np^{3/2}$. For $k>4$ a copy of $H$ is
{\it type two} if it contains a high degree pair $(x,y)$ belonging to
either $A'\times C'$ or $B'\times C'$; for $k=4$ we don't need this, and
simply declare that there
are no copies of type two.
\begin{lemma}\label{P3}
 W.l.p. G contains less than $2\vt \Psi(H,n)$ type two copies of $H$.
\end{lemma}

Set $s=\min\{t, n^{k-2}p^{\Cc{k-1}{2}}\}$,
the two regimes corresponding to the two ranges of $f(k,n,p)$
($= n^2p^{k-1}s$).
Define $\ww^*(\cdot)$ in the same way
as $\ww(\cdot)$, but with the count restricted to copies of
$H$ that are not type one or two.
Set
\beq{zeta}
\gz =\left\{\begin{array}{ll}
3^{k-2}\Psi(H,n,p)/(n^2p^{k-1}s)& \mbox{if $k\geq 5$}\\
225\Psi(H,n,p)/(n^2p^{3}s)&\mbox{if $k=4$}
\end{array}\right.
\enq

\mn
and (in either case) say
$ab\in \nabla(A,B)$ is {\em heavy} if
$\ww^*(a,b)>\gz.$
Finally, say a copy of $H$ is {\em type three} if it is not type one or two
and contains a heavy edge, and {\em type four} if it is not type one,
two or three.

\begin{lemma}\label{P4}
 W.l.p. G contains less
 than $4\vt \Psi(H,n)$ type three copies of $H$.
\end{lemma}

\begin{lemma}\label{P5}
 With probability at least $1-\exp[-\Omega_{\eps}(f(k,n,p))]$ G contains less
 than $(1+2\vt) \Psi(H,n)$ type four copies of $H$.
\end{lemma}

Of course Theorem \ref{GGThm}(a) (for $k\geq 4$) follows from
Lemmas \ref{P1}-\ref{P5}; these are proved in the next four sections.

\section{Proof of Lemma \ref{P1}}

\mn
For $i\in [3]$ set $D_1(i)=\{x\in V_i: d(x)> np^{2/5}\}$ and $D_2(i)=\{x\in V_i: np^{2/5}\geq d(x)> (1+\delta)np\}$, and for $j\in [2]$ set
$S_j(i)=\sum\limits\{d(x):x\in D_j(i)\}$.
We will show
\begin{prop}\label{wt(x)/d(x)}
For all $1\leq i\leq 3$,
\begin{eqnarray*}
\mbox{w.l.p. ~~~$\forall x\in D_j(i),$ ~~~} \ww(x)/d(x)<\left\{\begin{array}{ll}
2n^{k-2}p^{e_H-(k-1)} &\mbox{if $j=1$}\\
2n^{k-2}p^{e_H-k+2(k-1)/5} &\mbox{if $j=2$}
\end{array}\right.
\end{eqnarray*}
\end{prop}

\nin
and
\begin{prop}\label{S_j(i)}
For all $1\leq i\leq 3$,
\beq{Sji}
\mbox{w.l.p. ~~~$S_j(i)$}<\left\{\begin{array}{ll}
\vt n^2p^{k-1} &\mbox{if $j=1$}\\
kn^2p^{k-1}t &\mbox{if $j=2$.}
\end{array}\right.
\enq
\end{prop}

\mn
The lemma follows since
the number of type one copies of $H$ is at most
\begin{align*}
\sum_{x: x \text{high degree}} \ww(x)&<^* \sum_{i=1}^{3} (S_1(i)\cdot 2n^{k-2}p^{e_H-(k-1)} + S_2(i)\cdot 2n^{k-2}p^{e_H-k+2(k-1)/5})\\
&<^* 3(2\vt \Psi(H,n)+ 2k\Psi(H,n)p^{2(k-1)/5-1}t)\\
&< 7\vt \Psi(H,n),
\end{align*}
using Propositions \ref{wt(x)/d(x)} and \ref{S_j(i)} for the
first and second inequalities.
\qqqed

\bn
{\em Proof of Proposition} \ref{wt(x)/d(x)}.
Fix $i$ and condition on $\nabla(V_i)$ (thus determining $D_1(i)$ and
$D_2(i)$).
If $d_H(v_i)=k-1$, then for any $x\in D_1(i)$,
induction gives
$$\Pr(\ww(x)\geq 2\Psi(H_i,d(x)))<\exp [-\gO(f(k-1,d(x)))],$$
whence (noting $\Psi(H_i,\cdot)=\Psi(H_1,\cdot)$)
\begin{align}\label{firstf}
\Pr(\exists x\in D_1(i): \ww(x)\geq 2\Psi(H_1,d(x)))
&< n\exp [-\gO(f(k-1,np^{2/5}))]\nonumber\\
&<p^{n^2p^{k-1}}.
\end{align}
Similarly,
\begin{align}\label{secondf}
\Pr(\exists x\in D_2(i): \ww(x)\geq 2\Psi(H_1,np^{2/5}))
&<  n \Pr(X_{H_i}^{np^{2/5}}\geq 2\Psi(H_i,np^{2/5}))\nonumber\\
&< n\exp [-\gO(f(k-1,np^{2/5}))]\nonumber\\
&<p^{n^2p^{k-1}}
\end{align}
Note that, here and throughout, we omit the routine verifications
of inequalities like those in the last lines of
\eqref{firstf} and \eqref{secondf}.

\medskip
If $d(v_i)=k-2$, then
$v_i\not\sim v_j$ for some $j\in [k]$.  We
partition $V_j= P_1\cup \cdots \cup P_{\lfloor 1/p\rfloor}$ with
each $P_{\ell}$ of size at most $(1+\delta)np$,
and write $\ww ^{\ell}(x)$ for the number of copies of $H$ containing $x$
and meeting
$P_{\ell}$.  Noting that here
$\Psi(H_1,\cdot)=p^{-1}\Psi(H_i,\cdot)$
(and $\ww(x)=\sum_{\ell}\ww^\ell(x)$), we have
\begin{align*}
\Pr\left( \ww (x)\geq 2\Psi(H_1,d(x))\right)
&<\Pr(\exists \ell ~\ww^{\ell}  (x)\geq 2\Psi(H_i,d(x)))\\
&<p^{-1}\exp\left[-\gO( f(k-1,d(x)))\right]
\end{align*}
for a given $x$, so that
\begin{align}\label{thirdf}
\Pr\left(\exists x\in D_1(i): \ww (x)\geq 2\Psi(H_1,d(x))\right)
&< np^{-1} \exp\left[-\gO( f(k-1,np^{2/5}))\right]\nonumber\\
&< p^{n^2p^{k-1}},
\end{align}
and
\begin{align}\label{fourthf}
\Pr(\exists x\in D_2(i): \ww (x)\geq 2\Psi(H_1,np^{2/5}))
&< np^{-1}\Pr(X_{H_i}^{np^{2/5}}\geq 2\Psi(H_i,np^{2/5}))\nonumber\\
&< np^{-1}\exp[-\gO( f(k-1,np^{2/5}))]\nonumber\\
&< p^{n^2p^{k-1}}.
\end{align}

\mn
Finally, \eqref{firstf}-\eqref{fourthf} imply that w.l.p.
\begin{align*}
\ww(x)/d(x)< 2\Psi(H_1,d(x))/d(x)&= 2 (d(x))^{k-1}p^{e_H-(k-1)}/d(x)\\
&\leq 2n^{k-2}p^{e_H-(k-1)}  ~~~~~\forall x\in D_1(i)
\end{align*}
and
\begin{align*}
\ww(x)/d(x)< 2\Psi(H_1,np^{2/5})/d(x)
& = 2 (np^{2/5})^{k-1}p^{e_H-(k-1)}/d(x)\\
&\leq2n^{k-2}p^{e_H-k+2(k-1)/5}  ~~~~~\forall x\in D_2(i).
\end{align*}

\qed

\mn
{\it Proof of Proposition \ref{S_j(i)}.}
We bound $|\nabla(D_j(i))|$, which is, of course,
an upper bound on $S_j(i)$.
We first assert that, for any $i\in [3]$, w.l.p.
\beq{D_j(i)}
|D_1(i)|<
\vt np^{k-7/5} ~~\mbox{and}~~
|D_2(i)| < np^{k-2}t.
\enq
This will follow from Lemmas \ref{L1} and \ref{L2} (so really two applications
of Lemma \ref{L1}), a combination we will see repeatedly.
For a given $i$ and $j$ the events
$\{x\in D_j(i)\}$ ($x\in V_i$) are independent with (using Lemma \ref{L1})
$$
\Pr\left(x\in D_1(i)\right)< k
\Pr(B(n,p)> np^{2/5}) <
k(ep^{3/5})^{np^{2/5}}< p^{0.5np^{2/5}}
$$
and
$$
\Pr\left(x\in D_2(i)\right)<
k \Pr(B(n,p)> (1+\delta)np )
<\exp[-\gO(np)].
$$
An application of Lemma \ref{L2} now shows that (\ref{D_j(i)}) holds
w.l.p.\qqqed

\medskip
Assume then that \eqref{D_j(i)} holds, and for convenience rename its
bounds
$\vt np^{k-7/5}=r$ and $np^{k-2}t=u$;
we may of course assume $r\geq 1$ if proving the first bound in
\eqref{Sji} and $u\geq 1$ if proving the second.
We have (a bit crudely)
\begin{align*}
\Pr(|\nabla(D_1(i))|\geq \vt n^2p^{k-1}) &<
\Pr(\exists T\in \Cc{V(i)}{r}: |\nabla(T)|\geq \vt n^2p^{k-1})\\
&<\Cc{n}{r}  \Pr (B((k-1)rn,p)\geq \vt n^2p^{k-1})\\
&< n^r (e(k-1)p^{3/5})^{\vt n^2p^{k-1}}\\
&< p^{\gO_{\eps}(n^2p^{k-1})}
\end{align*}
and
\begin{align*}
\Pr(|\nabla(D_2(i))|\geq  kn^2p^{k-1}t)
&<   \Pr(\exists T\in \Cc{V(i)}{u}: |\nabla(T)|\geq kn^2p^{k-1}t)\\
&<
\Cc{n}{u}  \Pr(B((k-1)un,p)\geq k n^2p^{k-1}t)\\
&<n^u \exp[-\gO(n^2p^{k-1}t)]\\
&< p^{\gO(n^2p^{k-1})},
\end{align*}
with the third inequality in each case given by Lemma \ref{L1}.\qed

\section{Proof of Lemma \ref{P3}}

(Here we are only interested in $k\geq 5$.)
We bound the contribution of high-degree $(A',C')$-pairs,
the argument for $(B',C')$-pairs being similar.

\mn
Let $A''$ be the (random) set of vertices of $A'$ involved in
high-degree $(A',C')$-pairs---that is,
$A''=\{a \in A': \exists c\in C' ~d(a,c)>np^{3/2}\}$---and define $C''$ similarly.
We will show that
\beq{BC1}
\mbox{{\em w.l.p.} $~|A''|,|C''| < np^{k-5/2}$}
\enq
and
\beq{BC2}
\mbox{{\em w.l.p.} $~\ww(a,c) < 2t\Psi(H_{\{1,3\}},(1+\delta)np)
~~\forall (a,c)\in A'\times C'$.}
\enq
Combining these we find that the total weight of high degree
$(A',C')$-pairs is w.l.p. at most
$$
(np^{k-5/2})^2 2t\Psi(H_{\{1,3\}},(1+\delta)np)
~<~4n^2p^{3k-7}t \Psi(H_{\{1,3\}},n)
~<~\vt \Psi(H,n),
$$
where the second inequality uses
$\Psi(H_{\{1,3\}},n)\leq n^{-2}p^{-(2k-3)}\Psi(H,n)$ and
$4p^{k-4}t<\vt$ (see \eqref{pub}).
Since, as noted above, the
same argument shows that the contribution of high-degree
$(B',C')$-pairs is w.l.p.
at most $\vt \Psi(H,n)$, the lemma follows.

\mn
{\em Proof of \eqref{BC1}.}
Given $\nabla(C)$, the events $\{a\in A''\}$ are independent, with
\begin{eqnarray*}
\Pr\left(a\in A''\right)
&< &n(k-2)\Pr[B((1+\delta)np,p)>np^{3/2}]\\
&< & n(k-2)(e(1+\delta)p^{1/2})^{np^{3/2}} ~< ~p^{0.4np^{3/2}}~=: ~q,
\end{eqnarray*}

\mn
where we use \eqref{plb}, \eqref{pub} and $k\geq 5$ for the last inequality.
Thus, since $enq^{1/2}<1$, Lemma \ref{L2} gives
\eqref{BC1} for $A''$, and
of course the same argument applies to $C''$.\qqqed

\mn
{\em Proof of \eqref{BC2}.}
Here we have lots of room and just bound $\max\{\ww_3(a):a\in A'\}$,
a trivial upper bound on $\max\{\ww(a,c):a\in A', c\in C'\}$.
Since $d(a)<(1+\delta)np$ (for $a\in A'$) and $v_1\sim v_{\ell}$
$\forall \ell\in [k]\sm\{2,3\}$,
Theorem \ref{GGThm}(b) gives (inductively)
$$
\Pr[\exists a\in A' ~\ww_3(a)\geq 2t \Psi(H_{\{1,3\}},(1+\delta)np)]
~~~~~~~~~~~~~~~~~~~~~~~~~~~~~
$$
$$~~~~~~~~~~~~~~~< n\exp [-\gO( f(k-2,(1+\gd)npt^{\frac{1}{k-2}}))]
< p^{\gO(n^2p^{k-1})}
$$
(with verification of the second inequality, which does need
\eqref{plb} at one point, again left to the reader).\qqqed

\section{Proof of Lemma \ref{P4}}\label{PP4}

\mn
This requires special treatment when $k=4$;
see the beginning of Section \ref{k=4} for the reason for the split.
In Sections \ref{k=5} and \ref{k=4} we set
$A''=\{a: d_i(a)\leq (1+\delta)np ~\forall i\geq 3\}\supseteq A'$
and define $B''$ similarly.

\subsection{Proof for $k\geq 5$}\label{k=5}

\medskip
For reasons that will be explained as we proceed, we need
somewhat different arguments for large and small values of $p$.

\mn
{\em Case} 1: $np^{(k-1)/2}\geq \log^4 n$.
Let $C_b=\{c\in C\cap N(b): d(b,c)\leq np^{3/2}\}$ and
$$
W(A)~=~\{a: \exists b\in B'', \sum_{c\in C_b\cap N(a)} \ww_1(b,c)> \gz\}
~\supseteq ~\{a:\exists b,~\ww^*(a,b) > \gz\}
$$
(see \eqref{zeta} for $\gz$), and define $W(B)$ similarly.

\mn
{\em Remark.}
While it may seem more natural to define $W(A)$, $W(B)$ in
terms of $\ww(a,b)$ or $\ww^*(a,b)$, the present definition has
the advantage of not depending on $\nabla (A,B)$.
We will see something similar in Case 2.

\medskip
The point requiring most work here is
\beq{A''B''}
\mbox{w.l.p. $~~~|W(A)|, |W(B)| <
\vt np^{(k-1)/2}t^3.$}
\enq

\mn
Given this, the rest of the argument goes as follows.
According to
Lemma \ref{L1}, \eqref{A''B''} implies
\beq{nablaA''B''}
\mbox{w.l.p. $~~~|\nabla(W(A),W(B))| <
\vt n^2p^{k-1}$}
\enq
(since, given the inequality in \eqref{A''B''},
$|\nabla(W(A),W(B))|\sim B(m,p)$ for some $m < \vt^2 n^2p^{k-1}t^6$;
note the inequalities in \eqref{A''B''} and \eqref{nablaA''B''}
depend on separate sets of random edges).
On the other hand, an inductive application of Theorem \ref{GGThm}(b)
gives
\beq{w2*a}
\mbox{w.l.p. $~~~\ww^*(a,b) < 2\Psi(H_{\{1,2\}},(1+\delta)np)$
$~~~\forall a,b$}
\enq
(using the fact that we are in Case 1 and
noting that $d(a)>(1+\delta)np$ implies $\ww^*(a,b)=0$).

\mn
Finally, the combination of
\eqref{nablaA''B''} and \eqref{w2*a}
bounds the number of type three copies of $H$ by
$
\vt n^2p^{k-1}\cdot 2\Psi(H_{\{1,2\}},(1+\delta)np) < 4\vt\Psi(H,n).
$\qqqed

\medskip
Proofs of the two assertions in \eqref{A''B''} being similar,
we just deal with $W(A)$.
We first show
\beq{wlpbc}
\mbox{w.l.p. $~~~ \ww_1(b,c)<2tn^{k-3}p^{e_H-(3k-3)/2} =:\gc ~~~\forall b\in B''
~~\mbox{and}~ c\in C_b$}
\enq

\mn
and
\beq{wlpb}
\mbox{w.l.p. $~~~ \ww_1(b) < 4n^{k-2}p^{e_H-(k-1)}
~~~\forall b\in B''$.}
\enq
These will imply, via Lemma \ref{Chern}, that the events $\{a\in W(A)\}$ are
unlikely, and then \eqref{A''B''} will be an application of
Lemma \ref{L2}.

Each of \eqref{wlpbc} and \eqref{wlpb} is given (inductively) by
Theorem \ref{GGThm}(b), with small differences in arithmetic depending
on $d(v_2)$ and $d(v_3)$:  say we are
{\em in (a),(b)} or {\em (c)} according to whether
$(d(v_2),d(v_3))$ is
$(k-1,k-1)$, $(k-1,k-2)$ or $(k-2,k-2)$.

For \eqref{wlpbc} we first observe that,
given $\nabla(B\cup C)$ and $c\in C_b$,
$\ww_1(b,c)$ is stochastically
dominated by $X:=X(H_{\{1,2,3\}},np^{3/2})$
in (a) and (c),
and by the sum of $\lfloor 1/p\rfloor $ copies of
$X$   
in (b).
(For the latter assertion, let $\ell$ be the index for which
$v_3\not\sim v_\ell$ and, recalling that $b\in B''$, partition
$N(b)\cap V_\ell=V_1\cup\cdots \cup V_{\lfloor 1/p\rfloor}$ with each block
of size at most $np^{3/2}$.)
Theorem \ref{GGThm}(b) thus
gives the upper bound
\beq{fk-3}
n^2\lfloor 1/p\rfloor \exp[-\gO(f(k-3,np^{3/2}t^{1/(k-3)})]<
p^{\gO(n^2p^{k-1})}
\enq
on either
$$\Pr(\exists b\in B'',c\in C_b: ~\ww_1(b,c) >
2t \Psi(H_{\{1,2,3\}},np^{3/2})) $$
(if we are in (a) or (c)) or
$$\Pr(\exists b\in B'',c\in C_b: ~ \ww_1(b,c) >
2 t\lfloor 1/p\rfloor \Psi(H_{\{1,2,3\}},np^{3/2})) $$

\mn
(if we are in (b)),
the inequality in \eqref{fk-3} holding because we are in Case 1.
(Note that in \eqref{fk-3} the $\lfloor 1/p\rfloor$ is needed only when
we are ``in (b),"
and the term involving $t$ only when $k=5$.)

To complete the proof of \eqref{wlpbc} it just remains to check that $\gc$
(recall this is the right hand side of \eqref{wlpbc})
is an upper bound on
$2 t\Psi(H_{\{1,2,3\}},np^{3/2})$ if we are in (a) or (c),
and on $2 t\lfloor 1/p\rfloor \Psi(H_{\{1,2,3\}},np^{3/2})$
if we are in (b).\qqqed

\medskip
The proof of \eqref{wlpb} is similar.
Here, because we are in Case $1$, Theorem \ref{GGThm}(b) gives the bound
$$n\lfloor 1/p \rfloor\exp[-\gO(f(k-2,(1+\gd)np)]<p^{\gO(n^2p^{k-1})}$$
on
$\Pr(\exists b\in B'' ~\ww_1(b)>2\Psi(H_{\{1,2\}},(1+\delta)np))$
if we are in (a) or (b),
and on
$\Pr(\exists b \in B''~\ww_1(b)>2\lfloor 1/p\rfloor \Psi(H_{\{1,2\}},(1+\delta)np))$
if we are in (c);
and it's easy to check that $2\Psi(H_{\{1,2\}},(1+\delta)np)$
or $2\lfloor 1/p\rfloor \Psi(H_{\{1,2\}},(1+\delta)np)$
(as appropriate) is less than
$4n^{k-2}p^{e_H-(k-1)}$.\qqqed

\medskip
Finally we return to \eqref{A''B''}.
Fix (and condition on) any
value of
$E(G)\sm \nabla(A,C)$ satisfying the inequalities in
\eqref{wlpbc} and \eqref{wlpb}.
It is enough to show that, under this conditioning and for any $a$,

\beq{PraWA}
\Pr(a\in W(A))<
\exp[-\gO(np^{(k-1)/2}/t^2)]=:q,
\enq
since then Lemma \ref{L2} implies, using  $enq^{1/2}<1$ and
the fact that the events $\{a\in W(A)\}$ are independent,
$$|W(A)|<^*\vt np^{(k-1)/2}t^3.$$

\mn
(The assertion $enq^{1/2}<1$ (or $enq^c<1$) imposes
the most stringent requirement on $p$ for Case 1.)

For \eqref{PraWA} we observe that \eqref{wlpb} gives
(for any $a$ and $b\in B''$)

$$\E \sum_{c\in C_b\cap N(a)} \ww_1(b,c)=p\sum_{c\in C_b} \ww_1(b,c)
\leq p~ \ww_1(b)< 4n^{k-2}p^{e_H-k+2}
< \gz/2,$$
whence, using Lemma \ref{Chern} with \eqref{wlpbc}, we have
\begin{align*}
\Pr(a\in W(A))&<
\Pr\left(\exists b\in B''~~\sum\{\ww_1(b,c):c\in C_b\cap N(a)\} >\gz\right)\\
&<n\exp[-\gO(\gz/\gc)]
~<~n\exp[-\gO(np^{(k-1)/2}/t^2)]\\
&<\exp[-\gO(np^{(k-1)/2}/t^2)].
\end{align*}\qed

\mn
{\em Case} 2:  $np^{(k-1)/2}<\log^4 n$.
Recall that
for very small $p$---in particular for $p$ in the
present range---and $H\neq K_k$, Theorem \ref{GGThm} is contained in
Lemma \ref{lowerp};
we may thus assume $H=K_k$. Let
$H'=H-v_1v_2$ and, writing $\ww'
$ for $\ww_{H'}$, set
\beq{W(A)}
W(A)~=~\{a: \exists b\in B'',  \ww'(a,b)> \gz\}
~\supseteq~ \{a:\exists b~\ww^*(a,b) > \gz\},
\enq
and define $W(B)$ similarly.
(We could also work directly with $\ww(a,b)$
and avoid the extra definitions; but the present treatment,
which we will see again below, is more natural
in that it allows us to ignore
the essentially irrelevant $\nabla(A,B)$.)

The argument here is similar to that for Case 1.
We again show that membership in
$W(A)$, $W(B)$ is unlikely, leading to
\beq{A''B''2}
\mbox{w.l.p. $~~~|W(A)|, |W(B)| <
\log^8 n$},
\enq
which, in view of Lemma \ref{L1}, again gives
\beq{nablaA''B''2}
\mbox{w.l.p. $~~~|\nabla(W(A),W(B))| <
\vt n^2p^{k-1}.$}
\enq

On the other hand we will show, by an argument somewhat
different from others seen here,
\beq{wwab}
\mbox{w.l.p. $~\ww^*(a,b) < n^{k-2}p^{\binom{k-1}{2}}~~~\forall a,b.$}
\enq

\mn
Combining this with \eqref{nablaA''B''2} gives Lemma \ref{P4}
(for the present case).

\mn
{\em Proof of} \eqref{A''B''2}.
Of course it's enough to prove the assertion for $W(A)$.
We first observe that
\beq{ww1*b}
\mbox{w.l.p. $~\ww_1(b)<2t\Psi(H_{\{1,2\}},(1+\delta)np)
< 4t\log^{4k-8}n=:m  ~~\forall b \in B'';$}
\enq
as elsewhere, this is given by an inductive application of
Theorem \ref{GGThm}(b), which says that, for any $b \in B''$,
\begin{align*}
\Pr(\ww_1(b)> 2t\Psi(H_{\{1,2\}},(1+\delta)np))
&< \exp[-\gO(f(k-2,(1+\gd)npt^{1/(k-2)}))]\\
&< p^{\gO(n^2p^{k-1})}.
\end{align*}
(Note that for very small $p$
the extra factor $t$ in \eqref{ww1*b}---which did not appear in
\eqref{wlpb}---is needed for the final inequality here.)

We now condition on $E(G)\sm \nabla(A)$
and assume that, as in \eqref{ww1*b}, $\ww_1(b)< m ~\forall b\in B''$.
Note that $a\in W(A)$ means (at least) that there is some $b\in B''$ with
\beq{w*ab}
\ww'(a,b) \geq 3^{k-2}.
\enq

\mn
For $i\in \{3,\ldots, k\}$ (and any $b$), let
$V_i^*(b)$ be the set of vertices of $V_i$ lying on copies of $H_1$
that contain $b$.  Since
$$
\mbox{$\ww'(a,b) \leq \prod_{i=3}^k|N(a)\cap V_i^*(b)|$,}
$$
\eqref{w*ab} at least requires
$|N(a)\cap (\cup_{i=3}^k V_i^*(b))|\geq 3(k-2)$; so the probability
(for a given $a$) that there is some $b$ for which
\eqref{w*ab} holds is at most
$$
n\Pr(B((k-2)m,p)\geq 3(k-2))
<  p^{-(k-1)/2+(1-o(1))3(k-2)} < p^{k-1}=:q.
$$
But then, since (say) $enq^{3/4}<1$, Lemma \ref{L2} gives
\eqref{A''B''2}.\qqqed

\mn
{\em Remark.}
Of course \eqref{wwab} is the counterpart of \eqref{w2*a} of Case 1
(since $H$ is now $K_k$ the two bounds differ only by small constant factors);
but for very small $p$
the simple inductive derivation of \eqref{w2*a} using Theorem \ref{GGThm}(b)
no longer applies, since
$f(k-2,(1+\delta)np)$ may be much smaller than $ f(k,n)$.

\mn
{\em Proof of} \eqref{wwab}.
We may assume $b\in B'$ as otherwise $\ww^*(a,b)=0$.  For $i\in\{3,\ldots ,k\}$ let
$$V_i^*(a,b)= \{v\in V_i:  \mbox{some copy of $H$ on $a,b$ contains $v$}\}.$$
We will show that
\beq{abij}
\mbox{w.l.p. $~~|\nabla(V_i^*(a,b),V_j^*(a,b))| < n^2p^{k-1}
~~~\forall i,j,a ~\mbox{and} ~b\in B'.$}
\enq
That this gives \eqref{wwab} is essentially a special case of a
theorem of N. Alon
\cite{Alo}, the precise statement used here (see the proof of Theorem 1.1 in
\cite{FK}) being:
an $r$-partite graph
with at most $\ell$ edges between any two of its parts contains
at most $\ell^{r/2}$ copies of $K_r$.

For the proof of \eqref{abij} we fix $a,b$ and $i<j$, and
think of choosing edges of $G$ in the order:
(i)  $\nabla(b,V_3\cup\cdots\cup V_k)$;
(ii)  $\nabla(V_\ga,V_\gb)$ for all $3\leq \ga<\gb\leq k$ except
$(\ga,\gb) = (i,j)$;
(iii)  $\nabla(a, V_i\cup V_j)$;
(iv)   $\nabla(V_i,V_j)$.
(The remaining edges are irrelevant here.)

Let $H''=H_1-v_iv_j$.
Since $b\in B'$, Lemma \ref{lowerp} gives (since we are in Case 2)
\beq{wH'b}
\ww_{H''}(b) < ^* 2\Psi(H_{1,2}-v_iv_j,(1+\gd)np) =:m .
\enq
Let $V_i^*$ be the set of vertices of $V_i$
contained in copies of $H''$ that contain $b$, and define $V_j^*$ similarly.

If the bound in \eqref{wH'b} holds, then each of $V_i^*,V_j^*$ has
size at most $m< p^{-1}\log^{O(1)}n$; an application of
Lemma \ref{L1} thus shows that w.l.p. each of
$N(a)\cap V_i^*$, $N(a)\cap V_j^*$ (and thus also $V_i^*(a,b)$, $V_j^*(a,b)$)
has size at most (say) $p^{-1/4}$, and a second application gives
\eqref{abij}.\qed

\subsection{Proof for $k=4$}\label{k=4}

\mn
For $k=4$, as in Case 2 above, we can't simply
invoke induction to obtain \eqref{w2*a},
since $f(2,(1+\delta)np)$ ($\approx n^2p^3$) is smaller
than $f(4,n)$.
This is the main reason a separate argument
is needed for $k=4$.
\begin{proof}
In this section, for $x,y\in G$ let $d(x,y)=\max_{j\geq 3} d_j(x,y)$.
We consider the possibilities $H=K_4$ and $H=K_4^-$
($K_4$ with an edge removed) separately.

\bn
{\em Case} 1.  $H=K_4$.
Now $ab$ is heavy if
$\ww^*(a,b)>
225n^2p^3/s$.
Here it will be helpful to work with $\ww$ rather than $\ww^*$.
We treat (heavy) $ab$'s with
$\ww(a,b)>n^2p^3$ and those with $\ww(a,b)\in (225n^2p^3/s,n^2p^3]$
separately.

\medskip
To bound the contribution of edges of the first type,
set
$$A^*=\{a:\exists b\in B'', \ww'(a,b)>n^2p^{3}\}\supseteq \{a:\exists b\in B', \ww(a,b)>n^2p^{3}\}
$$
(where
$\ww'$ is as in the paragraph containing
\eqref{W(A)}),
and define $B^*$ similarly.
We first show
\beq{A*B*}
\mbox{w.l.p. $~~|A^*|, |B^*|<np^{7/4}$.}
\enq

\mn
To see this (for $A^*$, say) we
condition on the value of $\nabla(B,C\cup V_4)$ and consider
$\Pr(a\in A^*)$.
Noting that for any $a$ and $b\in B''$,
%
$$\Pr(\ww'(a,b)\geq n^2p^{3})\leq
\Pr(d(a,b)>np^{5/4})+\Pr(\ww'(a,b)\geq n^2p^{3}|d(a,b)\leq np^{5/4})
$$
(where $5/4$ is just a convenient value between $1$ and $3/2$), \
we have
\begin{eqnarray}\label{ainA*}
\Pr\left(a\in A^* \right)
&<&
n[ 2\Pr(B((1+\delta)np,p)>np^{5/4})
+ \Pr(B(n^2p^{5/2},p)>n^2p^3)]\nonumber\\
&\leq & p^{\gO(np^{5/4})}  + p^{\gO(n^2p^3)}.
\end{eqnarray}
Since (given $\nabla(B,C\cup V_4)$)
the events $\{a\in A^*\}$ are independent, Lemma \ref{L2} now gives
\eqref{A*B*}.
(Note that when the second term dominates \eqref{ainA*}, Lemma \ref{L2}
gives $A^*=\0$ w.l.p.)

On the other hand, again using Lemma \ref{L1}, we have
\begin{eqnarray*}
\Pr(\exists a\in A,b\in B': \ww(a,b)>n^2p^3t)&<& n^2\Pr(B((1+\delta)^2n^2p^2,p)>n^2p^3t)\\
&<& p^{\gO(n^2p^3)},
\end{eqnarray*}
and combining this with \eqref{A*B*} gives
$$
\sum\{\ww^*(a,b):
\ww(a,b)>n^2p^3 \}<^*|A^*||B^*|n^2p^3t <^* n^4p^{6.5}t
~~~(<\vt n^4p^6).$$

For $ab$ of the second type (i.e. with $\ww(a,b)\in (225n^2p^3/s,n^2p^3]$),
we take $J=15np^{3/2}/\sqrt{s}$, set
$A_J=\{a:\exists b\in B'', d(a,b)>J\}$, and define $B_J$ similarly.
Given $\nabla(B,C\cup V_4)$ the events $\{a\in A_J\}$ are independent with, for each $a$,
$$\Pr(a\in A_J)<2n\Pr(B((1+\delta)np,p)>J)<2n p^{(1-o(1))J/2}=:q.$$

\mn
(using $e(1+\delta)np^{3/2+o(1)}<J$ for the second inequality).
Since $enq^{1/2}<1$ (to see this, note $J$ is always at least 15,
and is $n^{\gO(1)}$ if $p>n^{-2/3+\gO(1)}$), Lemma \ref{L2} gives
$$|A_J|<^* \sqrt{\vt} n^2p^{3}/J.$$
Of course an identical discussion applies to
$|B_J|$, so we have $|A_J||B_J|<^* \vt sn^2p^3$ and, by Lemma \ref{L1},
$$|\nabla(A_J,B_J)|<^* \vt n^2p^3.$$
Thus, finally,
\begin{align*}
&\sum \{\ww^*(a,b):\text{$ab$ heavy, } \ww(a,b)\in (n^2p^3/s,n^2p^3]\}~~~~~~\\
&~~~~~~~~~~~~~~~~~~~~~~<^* |\nabla(A_J,B_J)|n^2p^3 = \vt n^4p^6
\end{align*}\qed

\bn
{\em Case} 2:  $H=K_4^-$.
Recall that $v_3v_4$ is the missing edge and an edge $ab$ is heavy if $\ww^*(a,b)>
225\Psi(H,n,p)/(n^2p^{3}s)=
225n^2p^2/s.$
We proceed more or less as in the second part of Case 1.

Set $J=15np/\sqrt{s}$,
$A_J=\{a: \exists b\in B'', d(a,b)>J\}$ and
$B_J=\{b: \exists a\in A'', d(a,b)>J\}$.
Given $\nabla(B,C\cup V_4)$
the events $\{a\in A_J\}$ are independent with,
for each $a$,
$$\Pr(a\in A_J)\leq 2n \Pr(B((1+\delta)np,p)>J)
<2np^{J/2}<p^{J/3}=:q~
$$

\mn
(using Lemma \ref{L1} and
$J>ep^{-1/2}(1+\delta)np^2$ for the second inequality).
Since (say) $enq^{1/2}<1$, Lemma \ref{L2} gives
$$|A_J|<^* n^2p^3/J,$$
and similarly for $B_J$.
Since $ab$ heavy at least requires $a\in A_J, b\in B_J$ and $a\in A'$
(and since $a\in A'$ implies $\ww(a,b) < ((1+\gd)np)^2$),
this says that
the number of type three copies of $H$
is
at most

$$|A_J||B_J|((1+\delta)np)^2<^*(n^2p^3/J)^2((1+\delta)np)^2 <\vt n^4p^5
$$
\end{proof}

\section{Proof of Lemma \ref{P5}}\label{LP5}

\mn
As earlier, set $H'=H-v_1v_2$ and $\ww'=\ww_{H'}$.  Let $X'=\sum_{a\in A, b\in B} \ww'(a,b)$.  Then $X'=X_{H'}$ depends only on $E(G)\sm \nabla(A,B)$.  Thus
\beq{Z}
X'<^* (1+\vt)\Psi(H',n)
=(1+\vt)\Psi(H,n)/p,
\enq
where the inequality is given by induction if $d(v_2)=k-1$
and by Lemma \ref{dk2} if $d(v_2)=k-2$.

\mn
Then
$$
Y:=\sum_{a\in A,b\in B}\min\{\ww'(a,b), \gz\}{\bf 1}_{\{ab\in E(G)\}}\geq
\sum_{a\in A,b\in B} \ww^*(a,b){\bf 1}_{\{{\sf w}^*(a,b)\leq \gz\}}.
$$
In view of \eqref{Z} it's enough to show that under any conditioning on
$E(G)\sm \nabla(A,B)$ for which
$X'<(1+\vt)\Psi(H,n)/p$,
$$
\Pr(Y>(1+2\vt)\Psi(H,n)) <
\exp [-\gO_{\vt}(n^2p^{k-1}s)]  ~~(=\exp[-\gO_\vt(f(k,n,p))]).
$$

\mn
But under any such conditioning (or any conditioning on $E(G)\sm \nabla(A,B)$),
the r.v.'s ${\bf 1}_{\{ab\in E(G)\}}$ are independent;
so, noting
$\E Y\leq pX'< (1+\vt) \Psi(H,n)$ and
using Lemma \ref{Chern}, we have
$$\Pr\left(Y>(1+2\vt)\Psi(H,n)\right) < \exp[-\gO_{\vt}(\Psi(H,n)/\gz)]=\exp [-\gO_{\vt}(n^2p^{k-1}s)].$$\qed

\section{Proof of Theorem \ref{lower}}\label{PT1.2}

Recall here $H=K_k$.
Set
$\rrr =\lceil 2\E \xi_H\rceil
=\lceil 2\binom{n}{k}p^{\binom{k}{2}}\rceil$.
Note that we only need to prove Theorem \ref{lower} for small $p$,
for simplicity say $p<n^{-2/(k-1)}\log n$, since above this
$f(k,n,p)= n^2p^{k-1}\log (1/p)$ and the theorem is given by the lower bound
in \eqref{jor}.
It will thus be enough to show

\begin{prop}
For $n^{-2/(k-1)}\leq p< n^{-2/(k-1)}\log n$,
$$\Pr(\xi_H=\rrr )> \exp[-O(\rrr )]$$
\end{prop}

\begin{proof}
(This is an easy generalization of the argument for $k=3$
given in \cite{DeM}.)
The number of sets $S$ of $\rrr $ vertex-disjoint copies of $H $ in $K_n$ is

\beq{srk}
s:=\frac{(n)_{\rrr k}}{\rrr !(k!)^{\rrr }} > \left(\frac{n^k}{\rrr  k^k}\right)^{\rrr }.
\enq

\mn
For such an $S$, let $Q_S$ and $R_S$ be the events
$\{G $ contains all members of $S\}$ and $\{S$ is the set of $H $'s of $G\}$.
We have
$\Pr(Q_S)=p^{\rrr \binom{k}{2}}$ and will show (for any $S$)
\beq{RSQS}
\Pr(R_S|Q_S) =\exp[-O(\rrr )],
\enq
whence (using \eqref{srk})
\begin{eqnarray*}
\Pr(\xi_H=\rrr ) &>& \sum_S \Pr(Q_S) \Pr(R_S|Q_S)
~=~s p^{\rrr \binom{k}{2}}\exp[-O(\rrr )]\\
&>& \left(\frac{n^kp^{\binom{k}{2}}}{\rrr  k^k}\right)^\rrr  \exp[-O(\rrr )]
~=~\exp[-O(\rrr )].
\end{eqnarray*}

\medskip
For the proof of \eqref{RSQS}, fix $S$; let $W$ be the union of the
vertex sets of the copies of $H $
in $S$; and
for $i=0\dots k$, let $T(i)$ be the set of $H $'s (in $K_n$)
having exactly $i$ vertices outside $W$.
We have
\begin{eqnarray}\label{PrRS}
\Pr(R_S|Q_S) &\geq &(1-p)^{|T(0)|}
\prod_{i=1}^k\left(1-p^{\C{i}{2}+(k-i)i}\right)^{|T(i)|}\\
&=& \exp[-O(\rrr )].\nonumber
\end{eqnarray}
Here the first inequality is given by Harris' Inequality \cite{Harris}
(which for our purposes says that for a product probability measure
$\mu$ on $\{0,1\}^E$ (with $E$ a finite set) and events $\A_i\sub \{0,1\}^E$
that are either all increasing or all decreasing,
$\mu(\cap \A_i)\geq \prod \mu(\A_i)$),
and for the second we can use, say,
$|T(i)|< n^i(\rrr k)^{k-i}$ for $0\leq i\leq k$.
(We omit the easy arithmetic, just noting that all factors but the last
(that is, $i=k$)
in \eqref{PrRS} are actually much larger than $\exp[-O(\rrr )]$.)

\end{proof}

\section{Concluding Remarks }\label{CR}
Of course the big question is, what is the true behavior of the probability
\eqref{Prxi} for general $H$?
We continue to use $\xi_H$ for $\xi_H^{n,p}$, and
here confine ourselves to $\eta =1$; that is,
we're interested in $\Pr(\xi_H>2\E \xi_H)$.
As usual we don't ask for more than the order of magnitude of
the exponent.

One can show, mainly following the argument of Section \ref{PT1.2},
that for any $K \subseteq H$
\beq{Elb}
\Pr\left(\xi_H\geq 2\E\xi_H\right)>\exp[-O_H(\Psi(K ,n,p))]
\enq
(where, recall, $\Psi(K,n,p)=n^{v_K}p^{e_K}$).
As far as we can see, it could be that
the truth in \eqref{Prxi} is
always given by the largest of the lower bounds in \eqref{Elb}
and \eqref{jor}.  For the latter we (finally) define
\beq{jor2}
M_H(n,p) =\left\{\begin{array}{ll}
n^2p^{\gD_H}
&\mbox{if $p\geq n^{-1/\gD_H}$}\\
\min_{K\sub H}(\Psi(K,n,p))^{1/\ga^*_K}
&\mbox{if $n^{-1/m_{_H}}\leq p\leq n^{-1/\gD_H}$}
\end{array}\right.
\enq

\mn
(where, as usual, $\ga^*$ is fractional independence number;
see e.g. \cite{JOR} or \cite{MGT}).
This is not quite the same as the quantity $M^*_H(n,p)$
used in \cite{JOR},
but, as shown in their Theorem 1.5, the two agree up to
a constant factor; so the difference is irrelevant here.

\begin{conj}\label{Conj}
For any $H$ and $p> n^{-1/m_H}$,
\beq{conj}
\Pr\left(\xi_H\geq 2\E\xi_H\right)
=\exp[-\Theta_H(\min\{\min_{K \sub H, e_K>0}\Psi(K ,n,p),M_H(n,p)t\})].
\enq
\end{conj}
\nin(Recall $t=\log(1/p)$.)  We remark without proof (it is not quite obvious as far as we know)
that, for a given $H$, the set of $p$ for which the (outer)
minimum in \eqref{conj} is $M_H(n,p)t$ is the interval
$[p_K,1]$, where $K$ is a smallest subgraph of $H$
with $m_K=m_H$ and $p_K$ is the unique $p$ for which
$\Psi(K,n,p)=M_H(n,p)\log (1/p)$.

Conjecture \ref{Conj}
gives a different perspective on the observation
from \cite[Section 8.1]{JOR} that $H=K_2$ shows that the lower bound
in \eqref{jor} is not always tight.
In this case $M_H(n,p)=n^2p$ for the full range of $p$ above and,
of course, $\xi_H$ is just
${\rm Bin}(\C{n}{2},p)$; so
the upper bound in \eqref{jor} is the truth.
But in fact \eqref{Elb} shows (with a little thought) that
the lower bound in \eqref{jor} is not tight
for {\em any} $H$ and
sufficiently small $p$ ($> n^{-1/m_H}$), since
for small enough $p$
one of the terms
$\Psi(K ,n,p)$ in \eqref{conj} is $o(M_H(n,p)t)$.
What's special about $K_2$ is that it is the only (connected) $H$ for
which the best lower bound is {\em never} given by \eqref{jor};
that is, the minimum in \eqref{conj} is never $M_H(n,p)t$.

\medskip
It also seems interesting to estimate
\beq{lld}
\Pr(\xi_H\geq \gc\E \xi_H)
\enq

\mn
when $\gc =\gc(n)=\go(1)$.
%
The present results essentially do this
for $H=K_k$ and ``generic" $p$; precisely,
Theorem \ref{GGThm}(b) implies (using a mild variant of
Proposition \ref{Equiv1})
\beq{ugnl}
\Pr(\xi_H>2\tau \Psi(H,n,p))<
\exp[-\Omega( f(k,n\tau^{1/k},p))],
\enq
which, for $p$ in the range where
$f(k,n\tau^{1/k},p)=n^2\tau^{2/k}p^{k-1}t $,
is (up to the constant in the exponent) the probability of containing
a clique of size
$np^{(k-1)/2}(2\tau)^{1/k}$
(provided this is not more than $\C{n}{k}$).
Of course
the trick that gets Theorem \ref{GGThm}(b)
from Theorem \ref{GGThm}(a) is general, so
results on Conjecture \ref{Conj} give corresponding upper bounds
for \eqref{lld}; but these bounds will not be tight in general,
and at this writing we don't have a good guess as to
the general truth in \eqref{lld}.\\

{\bf Acknowledgment.}
We would like to thank one of the referees for an exceptionally
careful reading and for pointing out \cite{Vu}.

\bn
Department of Mathematics\\
Rutgers University\\
Piscataway NJ 08854\\
rdemarco@math.rutgers.edu\\
jkahn@math.rutgers.edu

\end{document}